\documentclass[a4paper,12pt]{article}
\usepackage{amsfonts,amsthm,amsmath,amssymb,graphicx}

\title{Multifractal analysis of some multiple ergodic averages for the systems with non-constant Lyapunov exponents}

\date{}

\author{Lingmin Liao \thanks{partially supported by 12R03191A - MUTADIS (France).}\\
\normalsize LAMA UMR 8050, CNRS, Universit\'e Paris-Est Cr\'eteil,\\
\normalsize 61 Avenue du G\'en\'eral de Gaulle, 94010 Cr\'eteil Cedex, France\\
\normalsize e-mail: lingmin.liao@u-pec.fr\\
Micha\l \ Rams \thanks{supported by MNiSW grant N201 607640 (Poland). This paper was written during the visit of M.R. in Universit\'e Paris-Est Cr\'eteil.}\\
\normalsize Institute of Mathematics, Polish Academy of Sciences\\
\normalsize ul. \'Sniadeckich 8, 00-956 Warszawa, Poland\\
\normalsize e-mail: rams@impan.gov.pl}

\theoremstyle{plain}
\newtheorem{lem}{Lemma}[section]

\newtheorem{thm}[lem]{Theorem}

\theoremstyle{definition}

\theoremstyle{remark}
\newtheorem*{rem}{Remark}

\numberwithin{equation}{section}

\DeclareMathOperator{\diam}{diam}

\newcommand{\N}{\mathbb N}

\renewcommand{\epsilon}{\varepsilon}

\begin{document}

\maketitle

\def\thefootnote{}
\def\thefootnote{\arabic{footnote}}
\setcounter{footnote}{0}

\begin{abstract}
We study certain multiple ergodic averages of an iterated functions system generated by two contractions on the unit interval.  By using the dynamical coding $\{0,1\}^{\mathbb{N}}$ of the attractor, we compute the Hausdorff dimension of the set of points with a given frequency of the pattern $11$ in positions $k, 2k$.

\end{abstract}

\section{Introduction and statement of results}

Initiated by the paper of Fan Liao and Ma \cite{FLM}, the study of the multiple ergodic average from a point view of multifractal analysis have attracted much attention.  The major achievements have been made by Fan, Kenyon, Peres, Schmeling, Seuret, Solomyak, Wu and et al. (\cite{KPS11, FSW11, KPS12, PS, PeSo12,  FSW12, FanSW12, PSSS}). For a short history, we refer the readers to the paper of Peres and Solomyak \cite{PeSo12}.

Considered the symbolic space $\Sigma=\{0,1\}^{\mathbb{N}}$ with the metric $d(x,y)=2^{-\min\{n: \ x_n\neq y_n\}}$. In \cite{FLM}, the authors proposed to calculate the Hausdorff dimension spectrum of level sets of multiple ergodic averages. Among others, they asked the Hausdorff dimension of 
\begin{equation} \label{aa}
A_{\alpha}:=\Big\{(\omega_k)_1^{\infty}\in \Sigma : \lim_{n\to\infty} \frac{1}{n}\sum_{k=1}^{n} \omega_k\omega_{2k}=\alpha \Big\} \qquad (\alpha\in [0,1]).
\end{equation}
As a first step to solve the question, they also suggested to study a subset of $A_0$:
\begin{equation} \label{a}
A:=\Big\{(\omega_k)_1^{\infty}\in \Sigma : \omega_k\omega_{2k}=0 \quad \text{for all } k\geq 1\Big\}.
\end{equation}

 The Hausdorff dimension of $A$ was later given by Kenyon, Peres and Solomyak \cite{KPS12}.
\begin{thm}[Kenyon-Peres-Solomyak] We have
\[
\dim_H A=-\log (1-p),  \]
where $p\in [0,1]$ is the unique solution of the equation $$p^2=(1-p)^3.$$
\end{thm}
Enlightened by the idea of \cite{KPS12}, the question about $A_\alpha$ was finally answered by Peres and Solomyak \cite{PeSo12}, and independently by Fan, Schmeling and Wu \cite{FSW12}.
\begin{thm}[Peres-Solomyak, Fan-Schmeling-Wu]
For any $\alpha\in[0,1]$, we have
\[
\dim_H A_{\alpha}=-\log (1-p)- \frac{\alpha}{2} \log \frac{q(1-p)} {p(1-q)},
\] where $(p,q)\in [0,1]^2$ is the unique solution of the system % \[p^2(1-q)=(1-p)^3, \qquad \alpha=\frac {2pq} {2+p-q}.\]
 \[
 \left\{ 
 \begin{array}{ll}
\displaystyle p^2(1-q)=(1-p)^3,\\
 \displaystyle {2pq} =\alpha({2+p-q}). 
\end{array}\right.
\]
\end{thm}
We remark that a more general result on the Hausdorff dimension spectrum of level sets of multiple ergodic averages for a function depending only on one coordinate in $\Sigma$ has been obtained in \cite{FSW12}.

However, since the Lyapunov exponent is constant for the shift transformation on the symbolic space, what is obtained is in fact the entropy spectrum, i.e., the entropy (Bowen's definition see \cite{Bowen73}) of level sets of the multiple ergodic averages.%, with a constant difference\marginpar{\small what is constant difference?}. Our goal is to find the Hausdorff dimension spectrum for a simple system with nonconstant Lyapunov exponent. %We consider a two iterated function systems of the unit interval $[0,1]$.

Consider a piecewise linear map $T$ on the unit interval with two branches. Let $I_0, I_1\subset [0,1]$ be two closed intervals intersecting at most on one point. Let us also assume that $0\in I_0$ and $1\in I_1$. Suppose that on $I_0, I_1$, the map $T$ is bijective and linear onto $[0,1]$ with slops $e^{-\lambda_0}=1/|I_0|$ and $e^{-\lambda_1}=1/|I_1|$ ($\lambda_0, \lambda_1>0$) correspondingly.  Let
\[
J_T:=\cap_{n=1}^{\infty} T^{-n}[0,1].
\]
Then $(J_T, T)$ becomes a dynamical system. Similarly to \cite{FLM, PeSo12, FSW12},  We would like to study the following sets
\[
L:=\big\{x\in[0,1] : 1_{I_1}(T^k x) 1_{I_1}(T^{2k}x)=0, \quad \text{for all } k\big\},
\]
and
\[
L_\alpha:=\left\{x\in[0,1] : \lim_{n\to\infty}\frac{1}{n}\sum_{k=1}^{n}1_{I_1}(T^k x) 1_{I_1}(T^{2k}x)=\alpha\right\} \quad (\alpha \in [0,1]).
\]

For convenience, we will study a corresponding iterated function system and its natural coding. Let $\{f_0, f_1\}$ be an iterated function system on $[0,1]$  given by

\[
f_0(x) = e^{-\lambda_0}x,\ \  f_1(x) = e^{-\lambda_1}x + 1-e^{-\lambda_1}, \qquad (\lambda_0, \lambda_1>0)
\]
satisfying the open set condition, i.e., $e^{-\lambda_0}+e^{-\lambda_1}\leq 1$.
It has the usual symbolic description by $\Sigma = \{0,1\}^{\N}$ with a natural projection
\[
\pi(\omega) = \lim_{n\to\infty} f_{\omega_1}\circ f_{\omega_2}\circ\ldots\circ f_{\omega_n}(0).
\]
Let us define in $\Sigma$ the subsets $A$ and $A_\alpha$ by \eqref{aa}, \eqref{a}. Up to a countable set, the sets $L, L_\alpha$ can be written as
%\[
%S = \{\omega\in\Sigma: \omega_k \omega_{2k} = 0 \ \forall k\in \N\}
%\]
%and
%\[
%S_\alpha = \{\omega\in\Sigma: \lim_{n\to\infty} \frac 1 n \sum_{k=1}^n \omega_k \omega_{2k} = \alpha\}.
%\]
\[
L=\pi(A),  \quad L_\alpha = \pi(A_\alpha).
\]

We remark that if $\lambda_0=\lambda_1=\lambda$, i.e., the Lyapunov exponent is constant, then
\[   \dim_H L =\frac{\dim_HA}{\lambda/ \log 2}, \quad  \dim_H L_\alpha =\frac{\dim_HA_\alpha}{\lambda /\log 2}. \]
Furthermore, if $\lambda_0=\lambda_1=\log 2$, then $\pi(\Sigma)=[0,1]$, and the Hausdorff dimensions of $L, L_\alpha$ are the same as those of $A, A_\alpha$. %\marginpar{\small if $\lambda_1=\lambda_0\neq 1$ then we would need to change the metric on $\Sigma$}
Our goal is to calculate the Hausdorff dimension of sets $L$ and $L_\alpha$ for $\lambda_0\neq\lambda_1$.

Our results are as follows:
\begin{thm} We have
\[\dim_H L=\dim_H L_0 = -\frac {\log (1-p)} {\lambda_0},\] where $p\in[0,1]$ is the unique solution of the equation \[p^{2\lambda_0} = (1-p)^{2\lambda_1+\lambda_0}.\]
For any $\alpha\in(0,1]$, we have
\[
\dim_H L_{\alpha}=\frac {\alpha \log \frac{p(1-q)} {(1-p)q} -2\log (1-p)} {2\lambda_0},
\] where $(p,q)\in[0,1]^2$ is the unique solution of the system
 \[
 \left\{
 \begin{array}{ll}
\displaystyle \alpha (\lambda_1-\lambda_0) \log \frac {p(1-q)} {(1-p)q} +
 \lambda_0 \log \frac {p^2 (1-q)} {1-p} - 2\lambda_1 \log (1-p) =0,\\
 \displaystyle {2pq} =\alpha({2+p-q}).
\end{array}\right.
\]
\end{thm}

The paper is strongly related to \cite{PeSo12}, we mostly repeat the calculations there in a more complicated situation. For the lacking details, in particular for \cite[Lemma 2]{PeSo12} we refer the reader there. In the following two sections we calculate the lower bound: in Section 2 we introduce a family of measures and then we find the measure in this family that is supported on the set $L_\alpha$ and has maximal Hausdorff dimension, in Section 3 we find a formula for this dimension. In Section 4 we check that this formula is also the upper bound for the dimension of $L_\alpha$.

\bigskip

\section{Telescopic product measures}

The same measures that were used to calculate the entropy spectrum (see \cite{PeSo12}) will be useful for the Hausdorff spectrum as well.

Let us start from the multiplicative golden shift case. Given $p\in [0,1]$, let $\mu_p$ be a probability measure on $S$ given by
\begin{itemize}
\item[--] if $k$ is odd then $\omega_k=1$ with probability $p$,
\item[--] if $k$ is even and $\omega_{k/2}=0$ then $\omega_k=1$ with probability $p$,
\item[--] if $k$ is even and $\omega_{k/2}=1$ then $\omega_k=0$.
\end{itemize}
Precisely, let $(p_0, p_1):=(1-p, p)$ and let
\[
\begin{pmatrix}
  p_{00} & p_{01} \\
  p_{10} & p_{11}
\end{pmatrix}
:=\begin{pmatrix}
  1-p & p \\
  1 & 0
\end{pmatrix}.
\]
Then the measure $\mu_p$ of a cylinder is given by
\[
\mu_p([\omega_1\cdots\omega_{n}])=\prod_{k=1}^{\lceil n/2\rceil}p_{\omega_{2k-1}} \cdot \prod_{k=1}^{\lfloor n/2\rfloor} p_{\omega_{k}\omega_{2k}},
\]
where $\lceil \cdot \rceil, {\lfloor \cdot\rfloor}$ denote the ceiling function and the integer part function correspondingly.

Let $\nu_p = \pi_* \mu_p$. The Hausdorff dimension of $L$ will turn out to be the supremum of Hausdorff dimensions of $\nu_p$.

Similarly, to deal with the spectrum of the sets $L_\alpha$ we will define a family of probabilistic measures of two parameters. Given $p,q\in [0,1]$ we define a measure $\mu_{p,q}$ on $\Sigma$ as

\begin{itemize}
\item[--] if $k$ is odd then $\omega_k=1$ with probability $p$,
\item[--] if $k$ is even and $\omega_{k/2}=0$ then $\omega_k=1$ with probability $p$,
\item[--] if $k$ is even and $\omega_{k/2}=1$ then $\omega_k=1$ with probability $q$.
\end{itemize}
Similarly, if we let $(p_0, p_1):=(1-p, p)$ and let
\[
\begin{pmatrix}
  p_{00} & p_{01} \\
  p_{10} & p_{11}
\end{pmatrix}
:=\begin{pmatrix}
  1-p & p \\
  1-q & q
\end{pmatrix},
\]
then we have
\[
\mu_{p,q}([\omega_1\cdots\omega_{n}])=\prod_{k=1}^{\lceil n/2\rceil}p_{\omega_{2k-1}} \cdot \prod_{k=1}^{\lfloor n/2\rfloor} p_{\omega_{k}\omega_{2k}}.
\]

Once again, let $\nu_{p,q} = \pi_* \mu_{p,q}$. Please note that this notation is a little bit different from that in \cite{PeSo12}. Note also that $\mu_p = \mu_{p,0}$.

\begin{lem} \label{typical}
We have
\[
\mu_{p,q}(S_\alpha) =1
\]
for
\[
\alpha = \frac {2pq} {2+p-q}.
\]
\end{lem}
\begin{proof}
This lemma is proven in \cite[Lemma 3]{PeSo12}. However, we will need this proof as a starting point for the proof of Lemma \ref{dimension}.

Denote
\[
x_n(\omega) = \frac 2 n \sum_{k=n/2+1}^n \omega_k.
\]
For a $\mu_{p,q}$-typical $\omega$ the Law of Large Numbers implies
\[
x_{2n}(\omega) = \frac 1 2 p + \frac {x_n(\omega)} 2 q + \frac {1-x_n(\omega)} 2 p + o(1).
\]
Hence, as $k\to \infty$,
\[
x_{2^k n}(\omega) \to \frac {2p} {2+p-q}.
\]
By \cite[Lemma 5]{PeSo12}, it implies that $\mu_{p,q}$-almost surely

\begin{equation} \label{eq:freq}
\lim_{n\to\infty} x_n(\omega) =  \frac {2p} {2+p-q}.
\end{equation}
Then, for $\mu_{p,q}$-a.e. $\omega$,
\[
\frac 2 n \sum_{k=n/2+1}^n \omega_k \omega_{2k} = x_n(\omega) (q+o(1)) \to \frac {2pq} {2+p-q}.
\]
Thus the assertion follows.
\end{proof}

Let us denote

\[
H(p) = -p\log p -(1-p)\log(1-p)
\]
with convention $H(0)=H(1)=0$.

\begin{lem} \label{dimension}
We have

\[
\dim_H \nu_p = \frac {2H(p)} {2p\lambda_1 + (2-p)\lambda_0},
\]
and
\[
\dim_H \nu_{p,q}= \frac {(2-q)H(p) + pH(q)} {2p\lambda_1 + (2-p-q) \lambda_0}.
\]
\end{lem}
\begin{proof}
As $\nu_p = \nu_{p,0}$, it is enough to prove the second part of the assertion. For $\omega\in \Sigma$ denote

\[
C_n(\omega) = \{\tau\in \Sigma; \tau_k=\omega_k \ \forall k\leq n\}.
\]
Let

\[
h_n(\omega) := \log \mu_{p,q}(C_{2n}(\omega)) - \log \mu_{p,q}(C_n(\omega))
\]
and
\[
\lambda_n(\omega) := \log \diam \pi(C_{2n}(\omega)) - \log \diam \pi(C_n(\omega)).
\]
By the Law of Large Numbers, for $\mu_{p,q}$-typical $\omega$ and for big enough $n$ we have

\[
\frac 2 n h_n(\omega) = (2-x_n(\omega))(p\log p + (1-p)\log p) + x_n(\omega)(q\log q + (1-q)\log(1-q)) + o(1)
\]
and

\[
\frac 2 n \lambda_n(\omega) = (2-x_n(\omega))(-p\lambda_1 - (1-p)\lambda_0) + x_n(\omega)(-q\lambda_1 - (1-q)\lambda_0) + o(1).
\]
Thus, by \eqref{eq:freq}

\[
\frac {h_n(\omega)} {\lambda_n(\omega)} \to \frac {(2-q)H(p) + pH(q)} {2p\lambda_1 + (2-p-q)\lambda_0} \qquad \mu_{p,q}-a.e.
\]
 Hence, for $\mu_{p,q}$-a.e. $\omega$ we have

\[
\lim_{n\to \infty} \frac {\log \nu_{p,q}(\pi(C_n(\omega)))} {\log \diam \pi(C_n(\omega))} = \frac {(2-q)H(p) + pH(q)} {2p\lambda_1 + (2-p-q)\lambda_0}.
\]
\end{proof}

We will denote

\[
\gamma_\alpha = \left\{(p,q)\in [0,1]^2: \alpha = \frac {2pq} {2+p-q}\right\}.
\]

\begin{lem} \label{maximum}
The maximal Hausdorff dimension among measures $\nu_p$ is achieved for $p$ satisfying

\begin{equation} \label{eq:simple}
p^{2\lambda_0} = (1-p)^{2\lambda_1+\lambda_0}.
\end{equation}

For $\alpha\in (0,1)$, the maximal Hausdorff dimension among measures $\{\nu_{p,q}: (p,q)\in \gamma_\alpha\}$ is achieved for $(p,q)$ satisfying

\begin{equation} \label{eq:nice2}
\alpha (\lambda_1-\lambda_0) \log \frac {p(1-q)} {(1-p)q} + \lambda_0 \log \frac {p^2 (1-q)} {1-p} - 2\lambda_1 \log (1-p) =0.
\end{equation}

Such $(p,q)$ is unique in $\gamma_\alpha$ and is always in $(0,1)^2$.
\end{lem}
\begin{proof}
Let us start from the second part of assertion. We need to find the maximum of the function

\[
D(p,q) = \frac {(2-q)H(p) + pH(q)} {2p\lambda_1 + (2-p-q)\lambda_0}
\]
over the curve $\gamma_\alpha$. For $\alpha>0$ this curve's endpoints are $(1, 3\alpha/(2+\alpha))$ and $(\alpha/(2+\alpha),1)$. Moreover, we have

\[
\mathrm{d}\alpha = \frac 2 {(2+p-q)^2} (q(2-q)\mathrm{d}p + p(2+p)\mathrm{d}q).
\]
Hence, we need to solve the equation

\[
p(2+p) \frac {\partial D} {\partial p} - q(2-q) \frac {\partial D} {\partial q} =0.
\]
After expanding the left hand side and collecting the terms, it turns out that it is divisible by $p(2-q)$. We get

\begin{eqnarray} \label{eq:awful}
\begin{split}
& (2pq\lambda_1 + (4+2p-2q-2pq)\lambda_0) \cdot \log p  \\ +& ((-4-2p+2q-2pq)\lambda_1 + (-2-p+q+2pq)\lambda_0)  \cdot \log(1-p)\\ +& (-2pq\lambda_1 +2pq\lambda_0) \cdot  \log q \\ +& (2pq\lambda_1 +(2+p-q-2pq)\lambda_0) \cdot \log(1-q) =0.
\end{split}
\end{eqnarray}

It will be convenient to use $\beta = 2/\alpha$. As $(p,q)\in \gamma_\alpha$, we have

\[
2+p-q=\beta pq.
\]
Substituting this into \eqref{eq:awful}, we get

\begin{eqnarray} \label{eq:nice}
\begin{split}
&(2\lambda_1 + (2\beta-2)\lambda_0) \log p  +((-2\beta-2)\lambda_1 + (-\beta+2)\lambda_0)  \log(1-p) \\ +& (-2\lambda_1+2\lambda_0)  \log q + (2\lambda_1 + (\beta-2)\lambda_0) \log(1-q) =0
\end{split}
\end{eqnarray}
and \eqref{eq:nice2} follows.

To get the first part of assertion it is enough to remove all terms with $q$ and substitute $\alpha = 0$ into \eqref{eq:nice2}.

What remains is the third part of the assertion. Denoting by $F(p,q)$ the left hand side of \eqref{eq:nice}, we have

\[
F(1, 3\alpha/(2+\alpha)) = \infty
\]
and

\[
F(\alpha/(2+\alpha),1) = -\infty.
\]
We will check that $F$ restricted to $\gamma_\alpha$ is strictly monotone. We have

\[
p(p+2) \frac {\partial F} {\partial p} - q(2-q) \frac {\partial F} {\partial q} = \lambda_0 ((2\beta -2)(p+2)-2(2-q)) + {\rm spt},
\]
where spt stands for some positive terms (in particular, all the terms with $\lambda_1$ are positive). However, as
\[
(2\beta-2)(p+2) - 2(2-q) = 2p+2q+2(\beta-2)(p+2)>0,
\]
the coefficient for $\lambda_0$ is also positive.
%However, the non-omitted terms are also positive:
%\begin{itemize}
%\item[--] if $\beta\leq 3$ then $p\geq 1/2$ and $q\geq 3/4$, hence $p+2q\geq 2$,
%\item[--] if $\beta\geq 3$ then $2(\beta-2)\geq 2$.
%\end{itemize}
Hence, $F$ restricted to $\gamma_\alpha$ indeed has no extrema, so it must have only one zero.
\end{proof}

\begin{rem}
When $\alpha=0$, the curve $\gamma_0$ degenerates into two segments : $p=0$ and $q=0$. On the first segment, the dimension of $\dim_H\nu_{0,q}$ is zero. On the second segment, we have the assertion on $\nu_{p,0}=\nu_{p}$ in Lemma \ref{maximum}. When $\alpha=1$, the curve $\gamma_1$ degenerates into one point $(1,1)$, and we have $\dim_H\nu_{1,1}=0$.
\end{rem}

\begin{rem}
The curves $\gamma_\alpha$ cover whole $(0,1)^2$.
However, not all pairs $(p,q)\in (0,1)^2$ are solutions of \eqref{eq:nice} for any $\lambda_1, \lambda_0$. Indeed, we can write \eqref{eq:nice} in the form

\[
\frac {\lambda_1} {\lambda_0} a_1 + a_2 =0
\]
with
\[
a_1 = \alpha\log p + (-2-\alpha)\log (1-p) -\alpha\log q + \alpha\log (1-q)
\]
and
\[
a_2 = (2-\alpha)\log p + (\alpha-1)\log (1-p) +\alpha\log q + (1-\alpha)\log (1-q).
\]
Both $a_1$ and $a_2$ converge to $\infty$ as $p\to 1$ and to $-\infty$ as $q\to 1$. They are also both strictly monotone on $\gamma_\alpha$, which can be checked like in the third part of the proof of Lemma \ref{maximum} (using $(2-\alpha)(p+2)>\alpha(2-q)$ in case of $a_2$), so they both have unique zeros. As the equation

\[
ra_1+a_2=0
\]
can have positive solution only if $a_1$ and $a_2$ have different signs, only those $(p,q)\in \gamma_\alpha$ between zeros of $a_1$ and $a_2$, or equivalently satisfying

\[
\alpha \log \frac {p(1-q)} {(1-p)q} > \max \left(2\log (1-p), \log \frac {p^2 (1-q)} {1-p}\right),
\]
are solutions of \eqref{eq:nice} for some choice of $\lambda_1, \lambda_0$.
\end{rem}

\begin{rem}
The measures $\mu_{p,q}$ for $p=q$ are Bernoulli. Each $\gamma_\alpha$ intersects the diagonal $\{p=q\}$ in exactly one point $(\alpha^{1/2}, \alpha^{1/2})$ and at this point $a_1>0, a_2<0$. So, \eqref{eq:nice} has a Bernoulli measure as a solution for each $\alpha\in (0,1)$. It happens when

\[
\lambda_0 \log p = \lambda_1 \log (1-p),
\]
that is, when $\nu_{\alpha^{1/2},\alpha^{1/2}}$ is the Hausdorff measure (in dimension $\dim_H \pi(\Sigma)$) on $\pi(\Sigma)$.
\end{rem}

\section{Exact formulas}

To be able to provide the upper bounds in the following section, we need to substitute the results of Lemma \ref{maximum} to Lemma \ref{dimension} and obtain simpler formulas for our lower bound. We start with the golden shift case. Given $\lambda_1, \lambda_0$ let $p$ be given by \eqref{eq:simple}.

\begin{lem} \label{exact1}
We have

\[
\dim_H \nu_p = -\frac {\log (1-p)} {\lambda_0}.
\]
\end{lem}
\begin{proof}
By Lemma \ref{dimension},

\[
\dim_H \nu_p =\frac {2H(p)} {2p\lambda_1 + (2-p)\lambda_0}.
\]
Applying \eqref{eq:simple} it is easy to check that

\[
(2p\lambda_1+(2-p)\lambda_0)\log (1-p) = -2H(p) \lambda_0
\]
and the assertion follows.
\end{proof}

The calculations for the multifractal case are a little bit more complicated. Given $\lambda_1, \lambda_0$, and $\alpha$, let $p,q$ be given by \eqref{eq:nice2}.

\begin{lem} \label{exact2}
We have

\begin{equation} \label{eq:exact1}
\dim_H \nu_{p,q} = \frac {\alpha \log \frac{p(1-q)} {(1-p)q} -2\log (1-p)} {2\lambda_0}.
\end{equation}

If $\lambda_1 \neq \lambda_0$ then we have another formula:

\begin{equation} \label{eq:exact2}
\dim_H \nu_{p,q} = \frac {\log \frac {p^2 (1-q)} {(1-p)^3}} {2(\lambda_0 - \lambda_1)}.
\end{equation}

\end{lem}
\begin{proof}
By Lemma \ref{dimension},

\[
\dim_H \nu_{p,q} = \frac {(2-q)H(p) +pH(q)} {2p\lambda_1 + (2-p-q)\lambda_0}.
\]
Using \eqref{eq:nice2} one can check that
\[
(2p\lambda_1 + (2-p-q)\lambda_0)\left(\alpha \log \frac {p(1-q)} {(1-p)q} -2\log (1-p)\right) =2\lambda_0 ((2-q)H(p) +pH(q)).
\]
This gives \eqref{eq:exact1}. Applying \eqref{eq:nice2} once again we get

\begin{equation} \label{eq:aux}
\dim_H \nu_{p,q} = \frac {\alpha \log \frac {p(1-q)} {(1-p)q} + \log \frac {1-p} {p^2 (1-q)}} {2\lambda_1}.
\end{equation}
Together with \eqref{eq:exact1} this gives \eqref{eq:exact2}.
\end{proof}

%\begin{rem}
%Equations \eqref{eq:exact1} and \eqref{eq:exact2} do not easily show that our formulas reduce to ones given in \cite{PeSo} in case $\lambda_1=\lambda_0=1$. However, \eqref{eq:exact1} together with \eqref{eq:aux} easily give
%
%\[
%\dim_H \nu_{p,q} = \left( \alpha \log \frac {p(1-q)} {(1-p)q} -2\log (1-p)\right) \left(\frac 1 {4\lambda_0} + \frac 1 {4\lambda_1} \right) - \frac 1 {4\lambda_1} \log \frac {p^2 (1-q)} {(1-p)^3}.
%\]
%
%\end{rem}

\section{Upper bounds}

The last part of the proof is the upper bound.

\begin{lem}
We have

\[
\dim_H L \leq \sup_p \dim_H \nu_p,
\]
and for all $\alpha\in [0,1]$,
\[
\dim_H L_\alpha \leq \sup_{(p,q)\in \gamma_\alpha} \dim_H \nu_{p,q}.
\]
\end{lem}
\begin{proof}
As $L\subset L_0$, it is enough to prove the second part of the assertion. Fix $\alpha$ and let $\omega\in S_\alpha$. Let $p,q$ be as in \eqref{eq:nice2}. We denote for all $n\in\N$

\[
X_1^n = \sharp \{k\in [1,n]: \omega_k=1\}
\]
and for all even $n\in\N$
\[
X_{11}^n = \sharp \{k\in [1,n/2]: \omega_k = \omega_{2k} =1\}.
\]

We also denote

\[
\tilde{h}_n = -\log \mu_{p,q}(C_n(\omega))
\]
and
\[
\tilde{l}_n = -\log \diam \pi(C_n(\omega)).
\]

For any even $n$ we have (see \cite[Section 4]{PeSo12})

\[
-\tilde{h}_n = n\log (1-p) + X_1^{n/2} \log \frac {1-q} {1-p} + X_1^n \log \frac p {1-p} - X_{11}^n \log \frac {p(1-q)} {(1-p)q}.
\]
 We also have

\[
\tilde{l}_n = (\lambda_1-\lambda_0) X_1^n +n\lambda_0.
\]
Substituting \eqref{eq:exact1} and \eqref{eq:exact2} we get

\[
\tilde{l}_n \dim_H \nu_{p,q} = - \frac 12 X_1^n \log \frac {p^2 (1-q)} {(1-p)^3} + \frac n 2 \left(\alpha \log \frac {p(1-q)} {(1-p)q} -2\log (1-p)\right).
\]

Hence,

\[
\frac 1 n (\tilde{l}_n \dim_H \nu_{p,q} - \tilde{h}_n) = \left(\frac \alpha 2 - \frac {X_{11}^n} n\right) \log \frac {p(1-q)} {(1-p)q} +\frac 12 \left(\frac {X_1^{n/2}} {n/2} - \frac {X_1^n} n\right) \log \frac {1-q} {1-p}.
\]
As the first summand converges to 0 and the second telescopes,

\[
\liminf_{n\to\infty} \frac 1 n (\tilde{l}_n \dim_H \nu_{p,q} - \tilde{h}_n) \leq 0
\]
and we are done.
\end{proof}

\end{document}